%% file: n6_bounds.tex
\documentclass[11pt]{article}
\usepackage[utf8]{inputenc}
\usepackage[T1]{fontenc}
\usepackage{amsmath,amssymb,amsthm}
\usepackage[margin=1.1in]{geometry}
\usepackage{booktabs}
\usepackage{url}
\usepackage{tikz}
\usetikzlibrary{shapes.geometric}

\newtheorem{theorem}{Theorem}
\newtheorem{proposition}[theorem]{Proposition}
\newtheorem{remark}[theorem]{Remark}

\title{Improved bounds for the smallest\\ 4-chromatic graph of girth six}
\author{Glauco Rampone\\[2pt]
  \small\texttt{grampone@ethz.ch}}
\date{August 2026}

\begin{document}
\maketitle

\begin{abstract}
For integers $k,g\ge 3$ let $n_g(k)$ denote the minimum order of a graph with
chromatic number $k$ and girth at least $g$. Exoo and Goedgebeur (DMTCS 2019)
proved $26\le n_6(4)\le 66$, the upper bound being given by an explicit
$66$-vertex graph that has remained the smallest known $4$-chromatic graph of
girth~$6$. We improve both bounds to
\[
  29 \;\le\; n_6(4) \;\le\; 64.
\]
The upper bound is witnessed by an explicit $4$-chromatic graph of girth~$6$
on $64$ vertices with $152$ edges; it is vertex- and edge-critical, and its
automorphism group is cyclic of order~$8$ and acts semiregularly. The lower
bound is an exhaustive isomorph-free computation in the SAT modulo symmetries
framework with co-certificate learning, driven by the edge-density bound of
Liu and Postle for $4$-critical graphs of girth five; it re-derives the
bound $n_6(4)\ge 26$ of Exoo and Goedgebeur by a disjoint method and is
validated in both directions on the known values $n_4(4)=11$ and
$n_5(4)=21$.
We complement the bounds with structural obstructions: no smaller witness can
be obtained from either known witness by local modifications; no
$4$-chromatic Cayley graph of girth~$6$ exists on $54$--$63$ vertices (for
orders $59$ and $61$ no vertex-transitive witness exists at all); and no
witness on at most $63$ vertices admits a semiregular automorphism group
with two or three vertex orbits, for any finite group. Every known witness
of an $n_g(4)$ record with $g\ge 6$ is a lift of a small base graph along a
semiregular group action; the results above close the most symmetric part
of that regime below $64$ vertices.
All claimed properties of the new graph are verified by independent programs
and, in addition, formally certified in the Lean~4 proof assistant: the
non-$3$-colourability is established inside Lean by a formally verified
checker that re-validates a $219{,}532$-node refutation certificate produced
by an external search, with a machine-checked soundness theorem.
\end{abstract}

\noindent\textbf{Keywords:} chromatic number, girth, critical graphs,
Cayley graphs, voltage graphs, SAT modulo symmetries, formal verification,
Lean.\\
\textbf{MSC 2020:} 05C15, 05C38, 05C25, 05C85, 68V20.

\section{Introduction}

A classical theorem of Erd\H{o}s \cite{Erdos1959} states that there exist
graphs of arbitrarily large girth and chromatic number; the proof is
probabilistic, and explicit small witnesses are hard to find. Following
\cite{ExooGoedgebeur2019}, for integers $k,g\ge 3$ let $n_g(k)$ denote the
minimum number of vertices of a graph with chromatic number $k$ and girth at
least $g$.

For $k=4$ the first two values are classical: $n_4(4)=11$, attained by the
Gr\"otzsch graph \cite{Chvatal1974}, and $n_5(4)=21$, attained by exactly
$18$ graphs including the Brinkmann graph \cite{BrinkmannMeringer1997}
(the value and the count are due to Royle; see
\cite{ExooGoedgebeur2019}). For girth $6$, Exoo and Goedgebeur
\cite{ExooGoedgebeur2019} proved
\[
  26 \;\le\; n_6(4) \;\le\; 66,
\]
where the lower bound rests on an exhaustive computation and the upper bound
on an explicit $5$-regular graph of order $66$ with a semiregular automorphism
with $6$ orbits of length $11$. To the best of our knowledge no improvement of
either bound has been published since. (The regular variant of the
problem, the minimum order of an $r$-regular graph of girth $g$ and
chromatic number $\chi$, has recently been studied in
\cite{AraujoPardoEtAl2025}, where upper bounds on it are derived from
explicit graphs of girth $g$ and chromatic number $\chi$, so that smaller
such graphs translate into smaller bounds there.)

We improve both bounds.

\begin{theorem}\label{thm:main}
There exists a graph on $64$ vertices with girth exactly $6$ and chromatic
number $4$. Consequently $n_6(4)\le 64$.
\end{theorem}

\begin{theorem}\label{thm:lower}
Every graph of girth at least $6$ on at most $28$ vertices is $3$-colourable.
Consequently $n_6(4)\ge 29$.
\end{theorem}

Beyond the two theorems, the paper contributes structural obstructions and
machine-checkable evidence. Section~\ref{sec:obstructions} proves that no
smaller witness can be obtained from $G_{64}$ (or from the $66$-vertex graph
of \cite{ExooGoedgebeur2019}) by local modifications, that no witness on
$54$--$63$ vertices can be a Cayley graph (for orders $59$ and $61$, no
vertex-transitive witness exists at all), and that no witness on at most
$63$ vertices admits a semiregular automorphism group with two or three
vertex orbits, for any finite group. Since every known witness of an
$n_g(4)$ record with $g\ge 6$ is a lift of a small base graph along a
semiregular group action, these results close the most symmetric part of
the regime in which all such witnesses live. (The restriction to $g\ge 6$
is necessary: the Gr\"otzsch graph, with automorphism group of order $10$
on $11$ vertices, admits no nontrivial semiregular action.)

Theorem~\ref{thm:lower} closes the orders $26$, $27$ and $28$, which were
beyond the reach of \cite{ExooGoedgebeur2019}. It is obtained with the SAT
modulo symmetries framework of Kirchweger and Szeider
\cite{KirchwegerSzeider2024} together with co-certificate learning
\cite{KirchwegerPeitlSzeider2023}, both used as distributed; our
contribution is the encoding, in particular a reduction to $4$-critical
graphs that lets the edge-density bound of Liu and Postle
\cite{LiuPostle2017} enter as a propositional constraint, together with
the campaign and its audit trail. The change in cost is worth recording:
\cite{ExooGoedgebeur2019} spent about $2.5$ CPU-years to reach $25$
vertices, whereas the entire ladder $14\le n\le 28$ took about $113$ hours
on one core of a consumer machine.

Finally, every claim about $G_{64}$ is checked by several mutually
independent programs and, beyond that, certified in the Lean~4 proof
assistant (Section~\ref{sec:verification}). Verified certificate checkers
are well established; what we emphasize is the end result, unusual for
records of this kind: a reader can re-derive Theorem~\ref{thm:main}
mechanically in minutes, from a self-contained file whose only trusted
components are the Lean kernel and its compiled evaluator.

The paper is organized as follows. Section~\ref{sec:graph} defines $G_{64}$
and states its properties; Section~\ref{sec:search} documents the search
that found it; Section~\ref{sec:verification} describes the verification
and the Lean certification; Section~\ref{sec:lower} proves
Theorem~\ref{thm:lower}; Section~\ref{sec:obstructions} proves the
obstructions. All code, data, logs and the Lean formalization are available
at \url{https://github.com/glaucorampone/G64} and as ancillary files; the
graph itself is deposited in the House of Graphs \cite{HouseOfGraphs2023}.

\section{The graph}\label{sec:graph}

Let $G_{64}$ be the graph with vertex set $\mathbb{Z}_{64}=\{0,\dots,63\}$
whose edge set is the union of the orbits, under the map $v\mapsto v+8
\pmod{64}$, of the $20$ edges
\[
  \{i,\, i+t \bmod 64\}
\]
for the pairs $(i,t)$ listed in Table~\ref{tab:orbits}. Explicitly, the orbit
of the pair $(i,t)$ consists of the edges
$\{\,\{v,\,v+t \bmod 64\} : v\equiv i \pmod 8\,\}$; orbits with $t=32$ contain
$4$ distinct edges and all others contain $8$, for a total of $152$ edges.
The graph is drawn in Figure~\ref{fig:g64}.

\begin{table}[ht]
\centering
\begin{tabular}{llllllllll}
\toprule
$(0,15)$ & $(0,24)$ & $(0,49)$ & $(0,60)$ & $(1,5)$ &
$(1,22)$ & $(1,38)$ & $(1,42)$ & $(2,10)$ & $(2,32)$ \\
$(2,45)$ & $(2,52)$ & $(3,17)$ & $(3,32)$ & $(3,42)$ &
$(4,8)$ & $(5,1)$ & $(5,41)$ & $(5,42)$ & $(5,57)$ \\
\bottomrule
\end{tabular}
\caption{Edge-orbit representatives $(i,t)$ of $G_{64}$ under $v\mapsto v+8$.}
\label{tab:orbits}
\end{table}

\begin{proposition}\label{prop:props}
The graph $G_{64}$ is $4$-connected (indeed $\kappa=\lambda=\delta=4$), has
$64$ vertices, $152$ edges, $48$ vertices of degree $5$ and $16$ of degree $4$,
and satisfies:
\begin{enumerate}
\item[(i)] $G_{64}$ contains no cycle of length $3$, $4$ or $5$, and contains
  the $6$-cycle $(0,40,36,44,34,15)$; hence its girth is exactly $6$;
\item[(ii)] $G_{64}$ admits the proper $4$-colouring listed in
  Table~\ref{tab:col}, and admits no proper $3$-colouring; hence
  $\chi(G_{64})=4$;
\item[(iii)] $G_{64}$ is vertex-critical and edge-critical: for every vertex
  $v$ the graph $G_{64}-v$ is $3$-colourable, and for every edge $e$ the graph
  $G_{64}-e$ is $3$-colourable;
\item[(iv)] the independence number of $G_{64}$ is $24$;
\item[(v)] $\operatorname{Aut}(G_{64})$ is cyclic of order $8$, generated by
  the semiregular automorphism $v\mapsto v+8 \bmod 64$; consequently $G_{64}$
  has $8$ vertex orbits and $20$ edge orbits.
\end{enumerate}
\end{proposition}

Theorem~\ref{thm:main} follows from (i) and (ii). Item (iii) shows that
$G_{64}$ carries no smaller witness as a subgraph, and item (iv) shows that
non-$3$-colourability is not forced by a counting argument
($3\cdot 24 \ge 64$), which is one reason why short certificates of
$\chi\ge 4$ are not immediate for this graph. Item (v) records that the
symmetry exploited by the search of Section~\ref{sec:search} is the entire
symmetry of the graph, a point we return to in
Section~\ref{sec:obstructions}.

\begin{table}[ht]
\centering
\small
\begin{tabular}{l}
$c = (0, 0, 0, 0, 0, 0, 2, 0, 3, 0, 0, 0, 3, 0, 3, 3, 2, 1, 0, 0, 2, 1, 0, 2,$\\
$\phantom{c = (}3, 2, 1, 2, 1, 1, 0, 3, 0, 1, 2, 2, 2, 1, 0, 2, 3, 1, 3, 2, 0, 3, 2, 3,$\\
$\phantom{c = (}2, 3, 1, 3, 1, 1, 2, 2, 3, 0, 3, 0, 3, 1, 2, 3)$
\end{tabular}
\caption{A proper $4$-colouring $c$ of $G_{64}$ (colour of vertex $v$ at
position $v$, $v=0,\dots,63$), as verified in the Lean formalization.}
\label{tab:col}
\end{table}

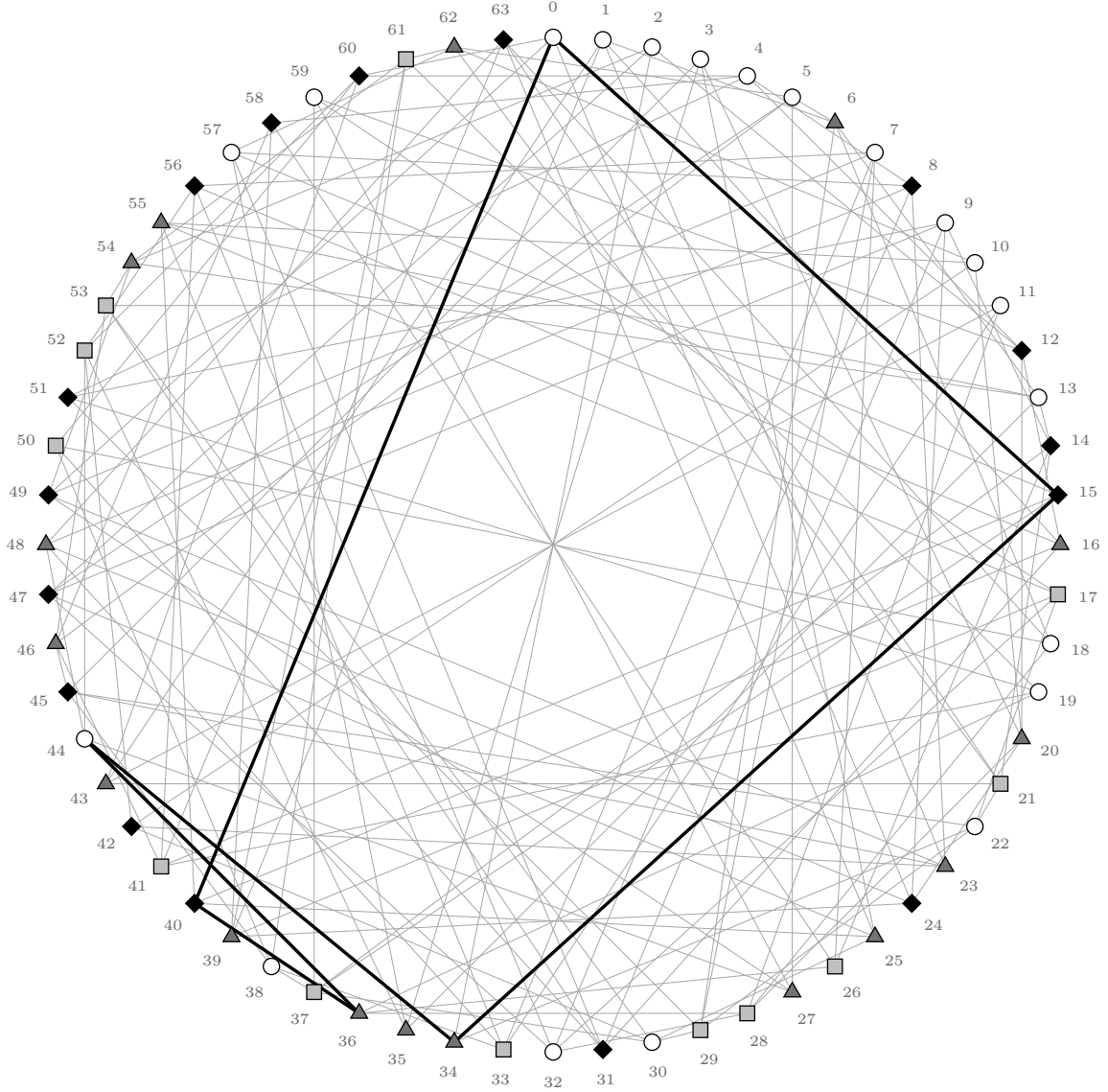
\begin{figure}[p]
\centering
\input{fig_g64}
\caption{The graph $G_{64}$, drawn with vertex $v$ at position $v$ on a
circle (vertex $0$ at the top, clockwise), so that each edge orbit of
Table~\ref{tab:orbits} appears as a chord pattern repeated under the
rotation $v\mapsto v+8$. Vertex shapes and shades encode the four colour
classes of Table~\ref{tab:col} (class $0$: white circle, $1$: light-grey
square, $2$: grey triangle, $3$: black diamond); the bold cycle is the
$6$-cycle $(0,40,36,44,34,15)$ of Proposition~\ref{prop:props}(i).
The figure is generated from \texttt{record64.json} by
\texttt{gen\_fig\_g64.py}.}
\label{fig:g64}
\end{figure}

\section{How the graph was found}\label{sec:search}

The graph was found by a randomized search over graphs with a semiregular
cyclic automorphism, in the spirit of the searches of
\cite{ExooGoedgebeur2019}, but sweeping a broader family of factorizations. For
a factorization $n=rs$, the search samples maximal sets of edge-orbits under
$v\mapsto v+r \pmod{rs}$ subject to girth $\ge 6$ and to a cap on the degree of
each vertex class, maintained incrementally: by symmetry, adding an orbit
creates a short cycle if and only if it creates one through the representative
edge, which is checked by a single bounded breadth-first search. Candidates
that are dense enough to be $4$-chromatic (by the bound of Kostochka and Yancey
\cite{KostochkaYancey2014}, a $4$-critical graph has at least $(5n-2)/3$ edges;
the deposited code now uses the sharper girth-$5$ bound $(5n+2)/3$ of Liu and
Postle \cite{LiuPostle2017}, which the candidate that produced $G_{64}$, with
$152$ edges, passes as well)
are filtered by randomized DSATUR and tabu $3$-colouring heuristics, and
survivors are decided exactly by a backtracking solver and the SAT solver
CaDiCaL \cite{CaDiCaL}. The search family is the one of
\cite{ExooGoedgebeur2019}; what we add is a wider sweep of factorizations
and the incremental girth test. We document the search here so that the
discovery is reproducible.

The campaign that produced $G_{64}$ consisted of sixteen buckets $r\times s$
with $54\le rs\le 65$. The eighteen runs that terminated normally sampled
$5\,828\,084$ candidates in about $6.2$ CPU-hours and left $366$ instances that
resisted the heuristics, none of them a witness; $G_{64}$ arose in the
factorization $64=8\times 8$ and is the only witness the campaign ever
produced. The search is seeded, so the discovery reproduces from a cold start
in under five minutes; the log of that run and the tally above are deposited
with the code.

It is perhaps worth noting that the $66$-vertex witness of
\cite{ExooGoedgebeur2019} has $s=11$ prime, while our sweep covered
composite values of $s$ as well, and the new witness arose at the fully
composite factorization $64=8\times 8$. Whether this regime is genuinely
richer or the observation is an accident of a single sample, we cannot say.

\section{Verification}\label{sec:verification}

None of the claims of Section~\ref{sec:graph} need be taken on trust: each is
decided by more than one program, and Theorem~\ref{thm:main} is in addition
reduced to the kernel of a proof assistant.

All properties in Proposition~\ref{prop:props} were verified by independent
computations: girth by direct enumeration of paths and common neighbourhoods;
non-$3$-colourability both by an exact forward-checking backtracking solver
and by CaDiCaL (unsatisfiability of the standard colouring encoding);
criticality by $64+152$ further exact colourability checks; the independence
number by branch and bound.

In addition, Theorem~\ref{thm:main} is \emph{formally certified} in the
Lean~4 proof assistant \cite{Lean4}, in a self-contained file using only the
Lean core library. The formalization states and proves, for the concrete
adjacency structure of $G_{64}$:
(a) there is no cycle of length $3$, $4$ or $5$ (and the explicit $6$-cycle
exists); (b) the colouring of Table~\ref{tab:col} is proper; (c) no proper
$3$-colouring exists. For (c), a direct verified exhaustive search in the
style of a static-order backtracking is not feasible (we measured static
search trees exceeding $2\cdot 10^{8}$ nodes over hundreds of vertex
orderings), so the formalization instead uses a \emph{refutation
certificate}: the branching tree of an external forward-checking solver
($219{,}532$ internal nodes: $209{,}002$ MRV branching nodes plus
$10{,}530$ domain-wipeout nodes; serialized to $1.76$ million tokens) is
re-traversed inside Lean by a certificate checker, and a machine-checked
soundness theorem (proved by structural induction, using no evaluation
axioms) states that acceptance of any certificate implies that no proper
$3$-colouring exists. The certificate itself is untrusted data: a corrupted
certificate would simply be rejected by the checker. Beyond the kernel, the
only trusted component is Lean's compiled evaluation mechanism
(\texttt{native\_decide}), used to run the checker and the finite girth and
colouring checks; each of these evaluations is corroborated by the
independent computations above.

The Lean sources, the certificate, the adjacency data (including a
\texttt{graph6} string) and the search code are available as ancillary files.
The graph is also deposited in the House of Graphs \cite{HouseOfGraphs2023} as
graph \texttt{57236} (\url{https://houseofgraphs.org/graphs/57236}), where its
invariants (girth, chromatic number, connectivity and automorphism group
size among them) were recomputed by an independent implementation.

\section{An improved lower bound: $n_6(4)\ge 29$}\label{sec:lower}

The lower bound $n_6(4)\ge 26$ of \cite{ExooGoedgebeur2019} was obtained by
generating, with a girth-pruned extension of \texttt{geng}, all graphs with
minimum degree $3$, maximum degree at most $6$ and girth at least $6$ on
$19$--$25$ vertices ($\approx 2.5$ CPU-years) and $3$-colouring them all,
combined with a separate argument showing that a $4$-vertex-critical graph
of girth at least $6$ with maximum degree at least $7$ must have at least
$26$ vertices. We extend the exhaustion to $28$ vertices with a different
and substantially faster method.

\emph{Reduction to $4$-critical graphs.} Suppose some graph of girth
$\ge 6$ on at most $28$ vertices is not $3$-colourable. Take a subgraph
$H$ that is edge-minimal subject to having chromatic number $\ge 4$, and
delete its isolated vertices. The result is $4$-critical: its chromatic
number is exactly $4$ (deleting an edge lowers the chromatic number by at
most one), deleting any edge leaves a $3$-colourable graph by
edge-minimality, and deleting any vertex $v$ does too, since $v$ is
incident to some edge $e$ and $H-v$ is a subgraph of $H-e$. It still has
girth $\ge 6$ and some order $n'\le 28$. Every $4$-critical graph has minimum
degree $\ge 3$, and, being of girth $\ge 5$, has at least $(5n'+2)/3$ edges
by the bound of Liu and Postle \cite{LiuPostle2017}. Moreover $n'\ge 14$:
in a graph of minimum degree $3$ and girth $6$, the balls of radius $2$
around the two endpoints of any edge are disjoint trees, so
$n'\ge 2(1+2+4)=14$ (the $(3,6)$-cage is the Heawood graph). It therefore
suffices to show that for each $14\le n\le 28$ there is no graph on $n$
vertices with girth $\ge 6$, chromatic number $\ge 4$, minimum degree
$\ge 3$ and at least $\lceil (5n+2)/3\rceil$ edges. Note that, unlike in
\cite{ExooGoedgebeur2019}, no upper bound on the maximum degree is imposed.

\emph{Method.} For each $n$ we ran an exhaustive isomorph-free search in
the SAT modulo symmetries (SMS) framework of Kirchweger and Szeider
\cite{KirchwegerSzeider2024}: a CDCL SAT solver (CaDiCaL \cite{CaDiCaL})
equipped with a propagator that prunes partial adjacency matrices which
cannot extend to lexicographically minimal representatives of their
isomorphism classes. Girth, minimum degree and edge-count constraints are
encoded propositionally; the co-$\mathrm{NP}$ condition ``no proper
$3$-colouring exists'' is handled by \emph{co-certificate learning}
\cite{KirchwegerPeitlSzeider2023}: each time the solver proposes a graph, a
second solver searches for a proper $3$-colouring, and any colouring found
is added as a clause excluding every graph coloured by it. The minimality
check was run with a recursion cutoff, which can only weaken the symmetry
breaking; this is sound for nonexistence results (at worst, more than one
representative per isomorphism class is explored). Concretely, when the
cutoff is reached the checker increments a counter and returns without
emitting a symmetry-breaking clause, so the graph is left in the search
space: the explored space is a superset of a set of canonical
representatives, and exhausting a superset without a model implies that no
model exists. Over the whole ladder the cutoff was reached exactly once, at
$n=28$, in $1$ of the $44{,}418{,}587$ calls to the minimality checker at
that order; the solver's own statistics record this and are reproduced in
the deposited output. It would affect an exact-count claim, which we do not make, but not
Theorem~\ref{thm:lower}.

\emph{What is ours.} The framework, the minimality propagator and the
co-certificate learning mechanism are due to
\cite{KirchwegerSzeider2024,KirchwegerPeitlSzeider2023} and were used as
distributed. Our contribution to Theorem~\ref{thm:lower} is the encoding
above (in particular the reduction to $4$-critical graphs, which lets the
Liu--Postle density bound enter as a propositional edge-count constraint and
dispenses with any cap on the maximum degree), together with the campaign
and the audit trail described below.

\emph{Result.} For every $14\le n\le 28$ the search terminates with an
empty enumeration. This establishes Theorem~\ref{thm:lower}. The exploratory
ladder $14\le n\le 26$ ran on a single consumer laptop (WSL, one core per
instance) in under three hours in total; the figures quoted here are instead
those of the audited rerun, in which every order was executed on one core of
an Apple~M4 under one pinned toolchain, so that the per-order costs are
mutually comparable. The two new orders dominate the cost: $n=27$ took
$37{,}533$ seconds of solver time and $n=28$ took $361{,}682$ seconds, just
over four days. The per-order cost is collected in Table~\ref{tab:cost}.

\begin{table}[ht]
\centering
\begin{tabular}{lrrrrr}
\toprule
$n$ & $24$ & $25$ & $26$ & $27$ & $28$ \\
\midrule
solver time (s) & $286$ & $1{,}598$ & $6{,}485$ & $37{,}533$ & $361{,}682$ \\
\bottomrule
\end{tabular}
\caption{Solver time per order on one core of an Apple~M4, from the audited
rerun; each order costs between four and ten times the previous one, the
ratio rising to $9.6$ at the last step. For scale, the enumeration of
\cite{ExooGoedgebeur2019} required about $2.5$ CPU-years to cover the orders
$19$--$25$. The two computations are not strictly comparable
(\cite{ExooGoedgebeur2019} enumerate all graphs of girth $\ge 6$ with minimum
degree $3$ and maximum degree at most $6$, whereas the search here is
confined to the denser $4$-critical candidates and imposes no degree cap),
but the difference in scale is what allows the exhaustion to be pushed three
orders further on consumer hardware.}
\label{tab:cost}
\end{table}

\emph{Validation.} The identical pipeline was checked in both directions
against known values. In the nonexistence direction, the runs for
$14\le n\le 25$ re-derive the bound $n_6(4)\ge 26$ of
\cite{ExooGoedgebeur2019} by a disjoint method. In the existence direction: with girth
constraint $4$ (triangle-free) the pipeline enumerates on $11$ vertices
exactly one $4$-chromatic graph (the Gr\"otzsch graph, whose uniqueness is
Chv\'atal's theorem \cite{Chvatal1974}) and none on $10$, a setting
already used as a benchmark for co-certificate learning in
\cite{KirchwegerPeitlSzeider2023}, whose Table~1 covers triangle-free
non-$3$-colourable graphs on $10$--$14$ vertices; with girth constraint $5$
it finds on $21$ vertices a $4$-chromatic graph of girth $5$ within $39$
seconds (consistent with $n_5(4)=21$), whose girth and
non-$3$-colourability we re-verified independently by breadth-first search
and exhaustive backtracking. The encodings themselves are validated by an
ancillary harness that regenerates each deposited CNF from the recorded
encoder command (checking it byte-identical to the deposited SHA-256 hash),
fixes complete edge-variable assignments of candidate graphs, and checks
that satisfiability always coincides with a direct combinatorial evaluation
of the encoded predicate; models of free solves are likewise decoded and
re-checked. Finally, the whole ladder $14\le n\le 26$ was rerun in
audited form on a second machine and operating system, and the two new
orders $n=27$ and $n=28$ were run there in audited form as well, with the
solver's own exit code recorded verbatim at every level and its full output
deposited alongside hashes of the pinned toolchain.

\begin{remark}\label{rem:trust}
Unlike Theorem~\ref{thm:main}, whose proof is formally certified in Lean,
Theorem~\ref{thm:lower} rests on the correctness of the SMS toolchain and
of our encoding. The framework can emit machine-checkable certificates
(DRAT proofs for the SAT reasoning and a witnessing permutation for each
learned symmetry-breaking clause \cite{KirchwegerSzeider2024}), and an
end-to-end pipeline that checks SMS runs inside Lean, with LRAT proofs
imported by reflection, has very recently been announced
\cite{KirchwegerManriqueSzeider2026,Szeider2026}. Producing and formally
checking a complete certificate for the $n=26$, $n=27$ and $n=28$ runs
along those lines is left as future work. The validation above, and the
agreement with \cite{ExooGoedgebeur2019} on $14$--$25$ vertices, are our
present evidence of correctness.
\end{remark}

\section{Obstructions below 64}\label{sec:obstructions}

Whether $n_6(4)<64$ remains open (the exhaustive bound of
Section~\ref{sec:lower} stops at $28$ vertices), but we can prove that a
smaller witness, if it exists, cannot be obtained from $G_{64}$ by local
modifications, cannot be a Cayley graph, and cannot admit any semiregular
group action with two or three vertex orbits
(Proposition~\ref{prop:slices}). All computational claims in this
section were verified exhaustively; the scripts, including a re-verification
of Proposition~\ref{prop:rigid} that runs in seconds, are among the
ancillary files.

\begin{proposition}[Rigidity of $G_{64}$]\label{prop:rigid}
\begin{enumerate}
\item[(i)] $G_{64}$ has diameter $4$. In particular no two vertices are at
  distance $\ge 5$, so no edge can be added to $G_{64}$ without creating a
  cycle of length at most $5$; the same holds for $G_{64}-v$ for every
  vertex $v$.
\item[(ii)] For every pair of vertices $\{u,v\}$ (there are $256$ orbits of
  pairs under the automorphism $v\mapsto v+8$), the graph $G_{64}-\{u,v\}$
  admits at most $3$ girth-preserving edge additions, and \emph{every}
  maximal sequence of girth-preserving additions results in a $3$-colourable
  graph.
\item[(iii)] For every pair of nonadjacent vertices $u,v$ there exists a
  $u$--$v$ path of length $3$, $4$ or $5$. Consequently every identification
  of two nonadjacent vertices of $G_{64}$ creates a cycle of length at most
  $5$; identifying two \emph{adjacent} vertices, that is, contracting an
  edge $uv$, yields a $3$-colourable graph instead, since every proper
  $3$-colouring of $G_{64}-uv$, which exists by
  Proposition~\ref{prop:props}(iii), satisfies $c(u)=c(v)$ (otherwise it
  would properly $3$-colour $G_{64}$) and hence descends to the quotient.
\end{enumerate}
Hence no $4$-chromatic graph of girth $6$ on fewer than $64$ vertices can be
obtained from $G_{64}$ by deleting up to two vertices and adding
girth-preserving edges, or by identifying two vertices.
\end{proposition}

We note, concerning (iii), that identifying two nonadjacent vertices $u,v$ of
a non-$3$-colourable graph always preserves non-$3$-colourability (a proper
$3$-colouring of the quotient is a proper $3$-colouring of the original graph
with $c(u)=c(v)$), and that paths of length $2$ collapse harmlessly to single
edges in the quotient; so a nonadjacent pair whose only short connections
have length $2$ would immediately yield a $63$-vertex witness. Item (iii)
rules this out. The same rigidity phenomena hold for the $66$-vertex graph of
\cite{ExooGoedgebeur2019}, which we re-verified: it has diameter $4$ and no
distance-$\ge 5$ pair survives any single vertex deletion.

The second obstruction concerns symmetry. Cayley graphs are a natural
candidate family here (the smallest Cayley witness found in
\cite{ExooGoedgebeur2019} has order $96$), and the witness $G_{64}$ admits
a \emph{semiregular} automorphism and, by
Proposition~\ref{prop:props}(v), nothing beyond it; but no witness on
$54$--$63$ vertices can be a Cayley graph:

\begin{proposition}\label{prop:cayley}
\begin{enumerate}
\item[(i)] Every Cayley graph of an abelian group with degree at least $3$
  has girth at most $4$: if $s,t$ are connection-set elements with
  $t\notin\{s,s^{-1}\}$, then $(e,\,s,\,s+t,\,t)$ is a $4$-cycle.
\item[(ii)] No Cayley graph of a (generalized) dihedral group
  $\mathrm{Dih}(A)$ has both girth $\ge 6$ and chromatic number $\ge 4$: a
  connection set containing a ``rotation'' $a\in A$ and a ``reflection''
  $\rho$ yields the $4$-cycle $(e,\,\rho,\,\rho a,\,a^{-1})$; a rotation-only
  set falls under (i); and a reflection-only set gives a bipartite graph.
\item[(iii)] For every remaining group, that is, every group of orders
  $54$--$63$ that is neither abelian nor generalized dihedral (including
  $A_5$, the dicyclic groups, the Frobenius groups
  $\mathbb{Z}_{11}\rtimes\mathbb{Z}_5$, $\mathbb{Z}_{19}\rtimes\mathbb{Z}_3$
  and $\mathbb{Z}_2^3\rtimes\mathbb{Z}_7$, the semidirect products
  $\mathbb{Z}_7\rtimes\mathbb{Z}_8$ and $\mathbb{Z}_7\rtimes\mathbb{Z}_9$
  (whose actions are not faithful, so they are not Frobenius), the remaining
  semidirect products of
  orders $54$, $56$ and $60$, and direct products of smaller dihedral,
  dicyclic and alternating groups with cyclic groups), an exhaustive
  enumeration of all connection sets of degrees $4$ to $7$ finds no Cayley
  graph with girth $\ge 6$ and $\chi\ge 4$. This range suffices: degrees
  $\le 3$ are excluded by Brooks' theorem \cite{Brooks1941}, since $K_4$
  has girth $3$ and odd cycles are $3$-colourable, while the Moore-type
  bound for girth $6$ (the balls of radius $2$ around the endpoints of an
  edge are disjoint, so $n\ge 2(d^2-d+1)$) already makes degree $7$
  impossible below $86$ vertices, so the enumerated range includes a
  margin. For groups of odd order the connection sets have even size, so
  only degrees $4$ and $6$ occur there.
\end{enumerate}
Consequently no $4$-chromatic Cayley graph of girth $\ge 6$ exists on
$54$--$63$ vertices. For $n=59$ and $n=61$, since every vertex-transitive
graph of prime order is a circulant \cite{Turner1967}, no vertex-transitive
witness exists at all.
\end{proposition}

Item (i) is folklore, and items (i) and (ii) fit a known pattern: the
closer a group is to being abelian, the smaller the girth of its Cayley
graphs; see \cite{ConderExooJajcay2010} for quantitative versions for
nilpotent and solvable groups in the context of cage constructions.

The enumeration behind (iii) is feasible because Cayley graphs are
vertex-transitive: the girth equals the length of a shortest cycle through
the identity, which can be tested locally and monotonically along a
backtracking search over connection sets. The quantification over
\emph{all} groups is itself machine-verified: an ancillary script
(\texttt{group\_census.py}) constructs explicit multiplication tables for
every isomorphism class of orders $54$--$63$: it derives the candidate
lists from Sylow-theoretic normal-subgroup arguments, checks the group
axioms in full, checks the classes pairwise non-isomorphic, and checks
their number against the classical classification \cite{BescheEickOBrien2002}
($55$ classes in total); it then maps each class to the lemma or to the
enumeration log that covers it. The search spaces collapse dramatically under the girth
constraint (for instance, of the $16{,}926$ girth-$\ge 6$ connection sets
of degree $5$ on the dihedral group of order $62$, every single one is
reflection-only and hence bipartite), and the non-bipartite girth-$6$
Cayley graphs that exist in this range are few (all of degree $4$,
e.g.\ $315$ graphs on $\mathbb{Z}_7\rtimes\mathbb{Z}_9$, except for
$40$ of degree $5$ on $F_{20}\times\mathbb{Z}_3$), and all are easily
$3$-colourable. This may be read as an explanation of why the
witnesses at both $64$ and $66$ vertices possess semiregular, but not
transitive, symmetry.

The third obstruction extends the second to exactly that semiregular
symmetry. For a finite group $H$ and an integer $r\ge 2$, consider the
graphs on $n=r|H|$ vertices admitting a semiregular action of $H$ with $r$
vertex orbits (``fibers''), equivalently lifts of an $r$-vertex base
multigraph with voltages in $H$ \cite{Gross1974}. For $H=\mathbb{Z}_s$
these are the cyclic lifts of Section~\ref{sec:search}, and $G_{64}$ itself
is the case $H=\mathbb{Z}_8$, $r=8$. Every such graph is a union of edge
\emph{classes} (orbits of vertex pairs under the action), so an entire
symmetry type is encoded by one Boolean variable per class, and ``does this
type contain a witness?'' becomes a satisfiability question: cycles of
length $3$ and $4$ are excluded by clauses through fiber representatives,
longer short cycles lazily from proposed models, and
non-$3$-colourability is enforced by co-certificate learning
\cite{KirchwegerPeitlSzeider2023} exactly as in Section~\ref{sec:lower}.
The search is exhaustive over the type: it terminates either by producing
a witness or by proving that none exists.

\begin{proposition}\label{prop:slices}
For every $n\le 63$, every finite group $H$ and every $r\in\{2,3\}$ with
$r|H|=n$, no graph on $n$ vertices with girth $\ge 6$ and chromatic number
$\ge 4$ admits a semiregular action of $H$ with $r$ vertex orbits.
Moreover, the same holds for the types $(H,r)$ with $r\ge 4$ listed in
Table~\ref{tab:slices}.
\end{proposition}

\begin{proof}[Proof (computational)]
For $n\le 28$ the statement follows from Theorem~\ref{thm:lower}, so let
$30\le n\le 63$ (a semiregular action with $2$ or $3$ orbits forces
$2\mid n$ or $3\mid n$, so $n=29$ is vacuous). This gives $106$ pairs
$(H,r)$: the $66$ isomorphism classes of orders $15$--$31$ at $r=2$ and
the $40$ classes of orders $10$--$21$ at $r=3$. The exhaustive search
described above, run with the additional constraint of minimum degree
$3$, terminates with an empty enumeration on every one of them. The
minimum-degree constraint is discharged by induction on $r$: vertex
degrees are constant along fibers, and a fiber of degree $\le 2$ is
disjoint from every $4$-critical subgraph, so a hypothetical witness
violating the constraint reduces, by deleting that fiber, to a witness
with the same action and one fiber fewer. From $(H,3)$ this lands in
$(H,2)$, covered by the same argument at order $2|H|\le 42$ (or by
Theorem~\ref{thm:lower} when $2|H|\le 28$); from $(H,2)$ it lands in the
Cayley case $(H,1)$ of order $|H|\le 31$: orders up to $28$ are excluded
by Theorem~\ref{thm:lower}, orders $29$ and $31$ by
Proposition~\ref{prop:cayley}(i), and order $30$ by
Proposition~\ref{prop:cayley}(i)--(ii) for $\mathbb{Z}_{30}$ and $D_{15}$
together with an exhaustive enumeration (as in
Proposition~\ref{prop:cayley}(iii), degrees $4$--$5$) for
$\mathbb{Z}_3{\times}D_5$ and $\mathbb{Z}_5{\times}D_3$. The same
descent discharges the constraint for the types of
Table~\ref{tab:slices}, whose fiber-deleted cases all land in the
$r\in\{2,3\}$ statement or below $29$ vertices. The quantification over
all isomorphism classes of each order is verified by the same census
script as in Proposition~\ref{prop:cayley}, which also maps every class
to its exhaustion log.
\end{proof}

\begin{table}[ht]
\centering\small
\begin{tabular}{rll}
\toprule
$n$ & $r$ & exhausted types with $r\ge 4$\\
\midrule
$40$ & $4$ & $(\mathbb{Z}_{10},4)$, $(D_5,4)$ (both groups of order $10$)\\
$36$ & $4$ & $(\mathbb{Z}_9,4)$, $(\mathbb{Z}_3^2,4)$ (both groups of order $9$)\\
$32$ & $4$ & $(\mathbb{Z}_8,4)$, $(\mathbb{Z}_2{\times}\mathbb{Z}_4,4)$,
             $(\mathbb{Z}_2^3,4)$, $(D_4,4)$, $(Q_8,4)$ (all five groups of
             order $8$)\\
$30$ & $5$ & $(\mathbb{Z}_6,5)$, $(D_3,5)$ (both groups of order $6$)\\
\bottomrule
\end{tabular}
\caption{The exhausted semiregular types with more than three orbits
(Proposition~\ref{prop:slices}); the types with $r\in\{2,3\}$ are all of
them, for every order, and are not listed. $D_m$ is the dihedral group of
order $2m$ and $Q_8$ the quaternion group. The per-type model counts are
recorded in the deposited logs.}
\label{tab:slices}
\end{table}

Since a semiregular action of $H$ restricts to a semiregular action of any
subgroup $K\le H$, with $[H:K]$ times as many orbits,
Proposition~\ref{prop:slices} has a corollary complementing
Proposition~\ref{prop:cayley} below order $54$: no witness on at most $63$
vertices is a Cayley graph of a group possessing a subgroup of index $2$
or $3$.

Three positive controls validate the machinery. On the type
$(\mathbb{Z}_{11},6)$ at $n=66$ the same search terminates in seven
minutes by \emph{finding} a witness, isomorphic to the $66$-vertex graph
of \cite{ExooGoedgebeur2019}; at girth $5$ on $(\mathbb{Z}_7,3)$ it
immediately finds a $21$-vertex non-$3$-colourable graph of girth $5$
(matching $n_5(4)=21$); and $G_{64}$ is verified to be a union of $20$
classes of the $(\mathbb{Z}_8,8)$ type. Runs at $r|H|\le 28$ terminate
empty, consistently with Theorem~\ref{thm:lower}. The proposition covers,
in particular, the hottest regions of our randomized searches: the types
$(\mathbb{Z}_{21},3)$ at $n=63$ and $(D_{15},2)$ at $n=60$, whose random
sampling had produced thousands of near-$4$-chromatic candidates, are
provably empty.

Every known witness of an $n_g(4)$ record with $g\ge 6$ ($G_{64}$, and
the graphs on $66$ and $171$ vertices of \cite{ExooGoedgebeur2019}, the
latter described there as LCF graphs) is a lift with a semiregular cyclic
group action, with $8$, $6$ and $9$ orbits
respectively. Proposition~\ref{prop:slices} closes the more symmetric end
of this regime below $64$ vertices completely (up to $3$ orbits, every
group and every order) and parts of the rest (Table~\ref{tab:slices});
what remains open at $n=63$ are the types $(\mathbb{Z}_9,7)$ and
$(\mathbb{Z}_3{\times}\mathbb{Z}_3,7)$, which remained undecided after
$7.4$ and $2.6$ million candidate models respectively, and
$(\mathbb{Z}_7,9)$ and $(\mathbb{Z}_3,21)$, which we did not attempt; the
analogous finding run on the $(\mathbb{Z}_8,8)$ type at $n=64$, which must
terminate with a witness, was likewise not run to completion.

\section{Concluding remarks}

The bounds on $n_6(4)$ now stand at
\[
  29 \;\le\; n_6(4) \;\le\; 64.
\]
The lower-bound computation of Section~\ref{sec:lower} might be pushed, by
the observed growth rate, one or two orders further (the $n=28$ level
already took just over four days on a single core, and the cost ratio rose
to $9.6$ at that step, so $n=29$ is a matter of weeks on modest hardware),
but not to the range
$54$--$63$ where a smaller witness, if any, is most plausibly to be found.
Our searches below $64$ vertices found no smaller witness. They comprised
over twenty-five million randomized candidates for $54\le n\le 63$, drawn
from semiregular cyclic factorizations and from lifts of small base graphs
with voltage assignments in non-abelian groups, together with asymmetric
local searches at $n=60,62,63$ warm-started from mutilated copies of both
known witnesses, and the exhaustive analyses of
Section~\ref{sec:obstructions}. Among these
candidates, roughly $10^4$ resisted randomized heuristic $3$-colouring and
were decided by exact solvers: every one of them is $3$-colourable. The
random searches are far from exhaustive, and we do not consider the
evidence sufficient to conjecture whether $n_6(4)=64$; what
Section~\ref{sec:obstructions} does establish is that a smaller witness,
if one exists, has essentially none of the symmetry that produced every
witness known so far.

We close with three natural directions. First, deciding the remaining
semiregular types at $n=63$ (notably $(\mathbb{Z}_9,7)$ and
$(\mathbb{Z}_3^2,7)$) would make the symmetry obstruction complete at that
order. Second, Theorem~\ref{thm:lower} could be brought to the same
standard of evidence as Theorem~\ref{thm:main} by producing and formally
checking certificates for the SMS runs (Remark~\ref{rem:trust}). Third,
the graph $G_{64}$, being $4$-critical with a transparent algebraic
description, may be of independent interest as a test case for colouring
algorithms and for formal verification of computational combinatorics.

\subsection*{Acknowledgements}
The search pipeline, the verification scripts and the Lean formalization were
developed with the assistance of an AI system (Claude, Anthropic); all final
artifacts are deterministic and independently checkable. We thank the authors
of \cite{ExooGoedgebeur2019} for their clear account of the state of the art,
and Andrea Nicastro for suggesting the use of Lean for the formal
verification.

\end{document}

%% file: fig_g64.tex
\begin{tikzpicture}[
  v0/.style={circle,draw=black,fill=white,inner sep=0pt,minimum size=6.4pt,line width=0.5pt},
  v1/.style={rectangle,draw=black,fill=black!25,inner sep=0pt,minimum size=5.8pt,line width=0.5pt},
  v2/.style={regular polygon,regular polygon sides=3,draw=black,fill=black!55,inner sep=0pt,minimum size=7.6pt,line width=0.5pt},
  v3/.style={diamond,draw=black,fill=black,inner sep=0pt,minimum size=7.2pt,line width=0.5pt},
  e/.style={line width=0.3pt,black!35},
  w/.style={line width=1.3pt,black}
]
\coordinate (p0) at (0.0000,7.0000);
\coordinate (p1) at (0.6861,6.9663);
\coordinate (p2) at (1.3656,6.8655);
\coordinate (p3) at (2.0320,6.6986);
\coordinate (p4) at (2.6788,6.4672);
\coordinate (p5) at (3.2998,6.1734);
\coordinate (p6) at (3.8890,5.8203);
\coordinate (p7) at (4.4408,5.4111);
\coordinate (p8) at (4.9497,4.9497);
\coordinate (p9) at (5.4111,4.4408);
\coordinate (p10) at (5.8203,3.8890);
\coordinate (p11) at (6.1734,3.2998);
\coordinate (p12) at (6.4672,2.6788);
\coordinate (p13) at (6.6986,2.0320);
\coordinate (p14) at (6.8655,1.3656);
\coordinate (p15) at (6.9663,0.6861);
\coordinate (p16) at (7.0000,0.0000);
\coordinate (p17) at (6.9663,-0.6861);
\coordinate (p18) at (6.8655,-1.3656);
\coordinate (p19) at (6.6986,-2.0320);
\coordinate (p20) at (6.4672,-2.6788);
\coordinate (p21) at (6.1734,-3.2998);
\coordinate (p22) at (5.8203,-3.8890);
\coordinate (p23) at (5.4111,-4.4408);
\coordinate (p24) at (4.9497,-4.9497);
\coordinate (p25) at (4.4408,-5.4111);
\coordinate (p26) at (3.8890,-5.8203);
\coordinate (p27) at (3.2998,-6.1734);
\coordinate (p28) at (2.6788,-6.4672);
\coordinate (p29) at (2.0320,-6.6986);
\coordinate (p30) at (1.3656,-6.8655);
\coordinate (p31) at (0.6861,-6.9663);
\coordinate (p32) at (0.0000,-7.0000);
\coordinate (p33) at (-0.6861,-6.9663);
\coordinate (p34) at (-1.3656,-6.8655);
\coordinate (p35) at (-2.0320,-6.6986);
\coordinate (p36) at (-2.6788,-6.4672);
\coordinate (p37) at (-3.2998,-6.1734);
\coordinate (p38) at (-3.8890,-5.8203);
\coordinate (p39) at (-4.4408,-5.4111);
\coordinate (p40) at (-4.9497,-4.9497);
\coordinate (p41) at (-5.4111,-4.4408);
\coordinate (p42) at (-5.8203,-3.8890);
\coordinate (p43) at (-6.1734,-3.2998);
\coordinate (p44) at (-6.4672,-2.6788);
\coordinate (p45) at (-6.6986,-2.0320);
\coordinate (p46) at (-6.8655,-1.3656);
\coordinate (p47) at (-6.9663,-0.6861);
\coordinate (p48) at (-7.0000,-0.0000);
\coordinate (p49) at (-6.9663,0.6861);
\coordinate (p50) at (-6.8655,1.3656);
\coordinate (p51) at (-6.6986,2.0320);
\coordinate (p52) at (-6.4672,2.6788);
\coordinate (p53) at (-6.1734,3.2998);
\coordinate (p54) at (-5.8203,3.8890);
\coordinate (p55) at (-5.4111,4.4408);
\coordinate (p56) at (-4.9497,4.9497);
\coordinate (p57) at (-4.4408,5.4111);
\coordinate (p58) at (-3.8890,5.8203);
\coordinate (p59) at (-3.2998,6.1734);
\coordinate (p60) at (-2.6788,6.4672);
\coordinate (p61) at (-2.0320,6.6986);
\coordinate (p62) at (-1.3656,6.8655);
\coordinate (p63) at (-0.6861,6.9663);
\draw[e] (p0) -- (p24);
\draw[e] (p0) -- (p49);
\draw[e] (p0) -- (p60);
\draw[e] (p1) -- (p6);
\draw[e] (p1) -- (p16);
\draw[e] (p1) -- (p23);
\draw[e] (p1) -- (p39);
\draw[e] (p1) -- (p43);
\draw[e] (p2) -- (p12);
\draw[e] (p2) -- (p34);
\draw[e] (p2) -- (p47);
\draw[e] (p2) -- (p54);
\draw[e] (p3) -- (p20);
\draw[e] (p3) -- (p25);
\draw[e] (p3) -- (p35);
\draw[e] (p3) -- (p45);
\draw[e] (p4) -- (p8);
\draw[e] (p4) -- (p12);
\draw[e] (p4) -- (p51);
\draw[e] (p4) -- (p58);
\draw[e] (p4) -- (p60);
\draw[e] (p5) -- (p6);
\draw[e] (p5) -- (p27);
\draw[e] (p5) -- (p46);
\draw[e] (p5) -- (p47);
\draw[e] (p5) -- (p62);
\draw[e] (p6) -- (p13);
\draw[e] (p6) -- (p18);
\draw[e] (p6) -- (p29);
\draw[e] (p7) -- (p26);
\draw[e] (p7) -- (p29);
\draw[e] (p7) -- (p33);
\draw[e] (p7) -- (p49);
\draw[e] (p7) -- (p56);
\draw[e] (p8) -- (p23);
\draw[e] (p8) -- (p32);
\draw[e] (p8) -- (p48);
\draw[e] (p8) -- (p57);
\draw[e] (p9) -- (p14);
\draw[e] (p9) -- (p24);
\draw[e] (p9) -- (p31);
\draw[e] (p9) -- (p47);
\draw[e] (p9) -- (p51);
\draw[e] (p10) -- (p20);
\draw[e] (p10) -- (p42);
\draw[e] (p10) -- (p55);
\draw[e] (p10) -- (p62);
\draw[e] (p11) -- (p28);
\draw[e] (p11) -- (p33);
\draw[e] (p11) -- (p43);
\draw[e] (p11) -- (p53);
\draw[e] (p12) -- (p16);
\draw[e] (p12) -- (p20);
\draw[e] (p12) -- (p59);
\draw[e] (p13) -- (p14);
\draw[e] (p13) -- (p35);
\draw[e] (p13) -- (p54);
\draw[e] (p13) -- (p55);
\draw[e] (p14) -- (p21);
\draw[e] (p14) -- (p26);
\draw[e] (p14) -- (p37);
\draw[e] (p15) -- (p37);
\draw[e] (p15) -- (p41);
\draw[e] (p15) -- (p57);
\draw[e] (p16) -- (p31);
\draw[e] (p16) -- (p40);
\draw[e] (p16) -- (p56);
\draw[e] (p17) -- (p22);
\draw[e] (p17) -- (p32);
\draw[e] (p17) -- (p39);
\draw[e] (p17) -- (p55);
\draw[e] (p17) -- (p59);
\draw[e] (p18) -- (p28);
\draw[e] (p18) -- (p50);
\draw[e] (p18) -- (p63);
\draw[e] (p19) -- (p36);
\draw[e] (p19) -- (p41);
\draw[e] (p19) -- (p51);
\draw[e] (p19) -- (p61);
\draw[e] (p20) -- (p24);
\draw[e] (p20) -- (p28);
\draw[e] (p21) -- (p22);
\draw[e] (p21) -- (p43);
\draw[e] (p21) -- (p62);
\draw[e] (p21) -- (p63);
\draw[e] (p22) -- (p29);
\draw[e] (p22) -- (p34);
\draw[e] (p22) -- (p45);
\draw[e] (p23) -- (p42);
\draw[e] (p23) -- (p45);
\draw[e] (p23) -- (p49);
\draw[e] (p24) -- (p39);
\draw[e] (p24) -- (p48);
\draw[e] (p25) -- (p30);
\draw[e] (p25) -- (p40);
\draw[e] (p25) -- (p47);
\draw[e] (p25) -- (p63);
\draw[e] (p26) -- (p36);
\draw[e] (p26) -- (p58);
\draw[e] (p27) -- (p44);
\draw[e] (p27) -- (p49);
\draw[e] (p27) -- (p59);
\draw[e] (p28) -- (p32);
\draw[e] (p28) -- (p36);
\draw[e] (p29) -- (p30);
\draw[e] (p29) -- (p51);
\draw[e] (p30) -- (p37);
\draw[e] (p30) -- (p42);
\draw[e] (p30) -- (p53);
\draw[e] (p31) -- (p50);
\draw[e] (p31) -- (p53);
\draw[e] (p31) -- (p57);
\draw[e] (p32) -- (p47);
\draw[e] (p32) -- (p56);
\draw[e] (p33) -- (p38);
\draw[e] (p33) -- (p48);
\draw[e] (p33) -- (p55);
\draw[e] (p35) -- (p52);
\draw[e] (p35) -- (p57);
\draw[e] (p37) -- (p38);
\draw[e] (p37) -- (p59);
\draw[e] (p38) -- (p45);
\draw[e] (p38) -- (p50);
\draw[e] (p38) -- (p61);
\draw[e] (p39) -- (p58);
\draw[e] (p39) -- (p61);
\draw[e] (p40) -- (p55);
\draw[e] (p41) -- (p46);
\draw[e] (p41) -- (p56);
\draw[e] (p41) -- (p63);
\draw[e] (p42) -- (p52);
\draw[e] (p43) -- (p60);
\draw[e] (p44) -- (p48);
\draw[e] (p44) -- (p52);
\draw[e] (p45) -- (p46);
\draw[e] (p46) -- (p53);
\draw[e] (p46) -- (p58);
\draw[e] (p48) -- (p63);
\draw[e] (p49) -- (p54);
\draw[e] (p50) -- (p60);
\draw[e] (p52) -- (p56);
\draw[e] (p52) -- (p60);
\draw[e] (p53) -- (p54);
\draw[e] (p54) -- (p61);
\draw[e] (p57) -- (p62);
\draw[e] (p61) -- (p62);
\draw[w] (p0) -- (p15);
\draw[w] (p0) -- (p40);
\draw[w] (p15) -- (p34);
\draw[w] (p34) -- (p44);
\draw[w] (p36) -- (p40);
\draw[w] (p36) -- (p44);
\node[v0] at (p0) {};
\node[v0] at (p1) {};
\node[v0] at (p2) {};
\node[v0] at (p3) {};
\node[v0] at (p4) {};
\node[v0] at (p5) {};
\node[v2] at (p6) {};
\node[v0] at (p7) {};
\node[v3] at (p8) {};
\node[v0] at (p9) {};
\node[v0] at (p10) {};
\node[v0] at (p11) {};
\node[v3] at (p12) {};
\node[v0] at (p13) {};
\node[v3] at (p14) {};
\node[v3] at (p15) {};
\node[v2] at (p16) {};
\node[v1] at (p17) {};
\node[v0] at (p18) {};
\node[v0] at (p19) {};
\node[v2] at (p20) {};
\node[v1] at (p21) {};
\node[v0] at (p22) {};
\node[v2] at (p23) {};
\node[v3] at (p24) {};
\node[v2] at (p25) {};
\node[v1] at (p26) {};
\node[v2] at (p27) {};
\node[v1] at (p28) {};
\node[v1] at (p29) {};
\node[v0] at (p30) {};
\node[v3] at (p31) {};
\node[v0] at (p32) {};
\node[v1] at (p33) {};
\node[v2] at (p34) {};
\node[v2] at (p35) {};
\node[v2] at (p36) {};
\node[v1] at (p37) {};
\node[v0] at (p38) {};
\node[v2] at (p39) {};
\node[v3] at (p40) {};
\node[v1] at (p41) {};
\node[v3] at (p42) {};
\node[v2] at (p43) {};
\node[v0] at (p44) {};
\node[v3] at (p45) {};
\node[v2] at (p46) {};
\node[v3] at (p47) {};
\node[v2] at (p48) {};
\node[v3] at (p49) {};
\node[v1] at (p50) {};
\node[v3] at (p51) {};
\node[v1] at (p52) {};
\node[v1] at (p53) {};
\node[v2] at (p54) {};
\node[v2] at (p55) {};
\node[v3] at (p56) {};
\node[v0] at (p57) {};
\node[v3] at (p58) {};
\node[v0] at (p59) {};
\node[v3] at (p60) {};
\node[v1] at (p61) {};
\node[v2] at (p62) {};
\node[v3] at (p63) {};
\node[font=\tiny,text=black!60] at (0.0000,7.4200) {0};
\node[font=\tiny,text=black!60] at (0.7273,7.3843) {1};
\node[font=\tiny,text=black!60] at (1.4476,7.2774) {2};
\node[font=\tiny,text=black!60] at (2.1539,7.1005) {3};
\node[font=\tiny,text=black!60] at (2.8395,6.8552) {4};
\node[font=\tiny,text=black!60] at (3.4978,6.5439) {5};
\node[font=\tiny,text=black!60] at (4.1223,6.1695) {6};
\node[font=\tiny,text=black!60] at (4.7072,5.7357) {7};
\node[font=\tiny,text=black!60] at (5.2467,5.2467) {8};
\node[font=\tiny,text=black!60] at (5.7357,4.7072) {9};
\node[font=\tiny,text=black!60] at (6.1695,4.1223) {10};
\node[font=\tiny,text=black!60] at (6.5439,3.4978) {11};
\node[font=\tiny,text=black!60] at (6.8552,2.8395) {12};
\node[font=\tiny,text=black!60] at (7.1005,2.1539) {13};
\node[font=\tiny,text=black!60] at (7.2774,1.4476) {14};
\node[font=\tiny,text=black!60] at (7.3843,0.7273) {15};
\node[font=\tiny,text=black!60] at (7.4200,0.0000) {16};
\node[font=\tiny,text=black!60] at (7.3843,-0.7273) {17};
\node[font=\tiny,text=black!60] at (7.2774,-1.4476) {18};
\node[font=\tiny,text=black!60] at (7.1005,-2.1539) {19};
\node[font=\tiny,text=black!60] at (6.8552,-2.8395) {20};
\node[font=\tiny,text=black!60] at (6.5439,-3.4978) {21};
\node[font=\tiny,text=black!60] at (6.1695,-4.1223) {22};
\node[font=\tiny,text=black!60] at (5.7357,-4.7072) {23};
\node[font=\tiny,text=black!60] at (5.2467,-5.2467) {24};
\node[font=\tiny,text=black!60] at (4.7072,-5.7357) {25};
\node[font=\tiny,text=black!60] at (4.1223,-6.1695) {26};
\node[font=\tiny,text=black!60] at (3.4978,-6.5439) {27};
\node[font=\tiny,text=black!60] at (2.8395,-6.8552) {28};
\node[font=\tiny,text=black!60] at (2.1539,-7.1005) {29};
\node[font=\tiny,text=black!60] at (1.4476,-7.2774) {30};
\node[font=\tiny,text=black!60] at (0.7273,-7.3843) {31};
\node[font=\tiny,text=black!60] at (0.0000,-7.4200) {32};
\node[font=\tiny,text=black!60] at (-0.7273,-7.3843) {33};
\node[font=\tiny,text=black!60] at (-1.4476,-7.2774) {34};
\node[font=\tiny,text=black!60] at (-2.1539,-7.1005) {35};
\node[font=\tiny,text=black!60] at (-2.8395,-6.8552) {36};
\node[font=\tiny,text=black!60] at (-3.4978,-6.5439) {37};
\node[font=\tiny,text=black!60] at (-4.1223,-6.1695) {38};
\node[font=\tiny,text=black!60] at (-4.7072,-5.7357) {39};
\node[font=\tiny,text=black!60] at (-5.2467,-5.2467) {40};
\node[font=\tiny,text=black!60] at (-5.7357,-4.7072) {41};
\node[font=\tiny,text=black!60] at (-6.1695,-4.1223) {42};
\node[font=\tiny,text=black!60] at (-6.5439,-3.4978) {43};
\node[font=\tiny,text=black!60] at (-6.8552,-2.8395) {44};
\node[font=\tiny,text=black!60] at (-7.1005,-2.1539) {45};
\node[font=\tiny,text=black!60] at (-7.2774,-1.4476) {46};
\node[font=\tiny,text=black!60] at (-7.3843,-0.7273) {47};
\node[font=\tiny,text=black!60] at (-7.4200,-0.0000) {48};
\node[font=\tiny,text=black!60] at (-7.3843,0.7273) {49};
\node[font=\tiny,text=black!60] at (-7.2774,1.4476) {50};
\node[font=\tiny,text=black!60] at (-7.1005,2.1539) {51};
\node[font=\tiny,text=black!60] at (-6.8552,2.8395) {52};
\node[font=\tiny,text=black!60] at (-6.5439,3.4978) {53};
\node[font=\tiny,text=black!60] at (-6.1695,4.1223) {54};
\node[font=\tiny,text=black!60] at (-5.7357,4.7072) {55};
\node[font=\tiny,text=black!60] at (-5.2467,5.2467) {56};
\node[font=\tiny,text=black!60] at (-4.7072,5.7357) {57};
\node[font=\tiny,text=black!60] at (-4.1223,6.1695) {58};
\node[font=\tiny,text=black!60] at (-3.4978,6.5439) {59};
\node[font=\tiny,text=black!60] at (-2.8395,6.8552) {60};
\node[font=\tiny,text=black!60] at (-2.1539,7.1005) {61};
\node[font=\tiny,text=black!60] at (-1.4476,7.2774) {62};
\node[font=\tiny,text=black!60] at (-0.7273,7.3843) {63};
\end{tikzpicture}